\documentclass{proc-arxiv}

\newtheorem{theorem}{Theorem}[section]

\newtheorem{pro}[theorem]{Proposition}
\newtheorem{lem}[theorem]{Lemma}
\newtheorem{coro}[theorem]{Corollary}

\theoremstyle{definition}
\newtheorem{de}[theorem]{Definition}
\theoremstyle{remark}
\newtheorem{re}[theorem]{Remark}
\numberwithin{equation}{section}
\usepackage{amsmath}
\usepackage{amssymb}
\usepackage{amsthm}
\usepackage[utf8]{inputenc}
\usepackage[normalem]{ulem}
\usepackage{hyperref}
\usepackage{mathtools}
\mathtoolsset{showonlyrefs}
\renewcommand{\Re}{\operatorname{Re}}
\renewcommand{\Im}{\operatorname{Im}}
\usepackage[top=3cm, bottom=2cm, left=3cm, right=3cm]{geometry}
\usepackage[abbrev]{amsrefs}
\renewcommand{\MR}[1]{}

\begingroup 
\makeatletter
   \@for\theoremstyle:=theorem,de,re,pro,lem,coro,plain\do{%
     \expandafter\g@addto@macro\csname th@\theoremstyle\endcsname{%
        \addtolength\thm@preskip\parskip
     }%
   }
\endgroup

\begin{document}

\title[Scattering in the radial INLS]{Scattering of radial solutions to the Inhomogeneous Nonlinear Schrödinger Equation}

\author[L. Campos]{Luccas Campos}
\address{Department of Mathematics, UFMG, Brazil and Department of Mathematics, FIU, USA}
\email{luccasccampos@gmail.com}
\thanks{The author thanks Luiz Gustavo Farah (UFMG) and Svetlana Roudenko (FIU) for their valuable comments and suggestions which helped improve the manuscript. This work was done when the first author was visiting Florida International University in 2018-19 under the support of Coordenação de Aperfeiçoamento de Pessoal de Nível Superior - Brasil (CAPES), for which the author is very grateful as it boosted the energy into
the research project.
L. C. was financed in part by the Coordenação de Aperfeiçoamento de Pessoal de Nível Superior - Brasil (CAPES) - Finance Code 001.}

\begin{abstract}
We prove scattering below the mass-energy threshold for the focusing inhomogeneous nonlinear Schrödinger equation
		\begin{equation}
			iu_t + \Delta u + |x|^{-b}|u|^{p-1}u=0,\\
		\end{equation}
		when $b \geq 0$ and $N > 2$ in the intercritical case $0 < s_c <1$. This work generalizes the results of Farah and Guzm\'an \cite{FG_Scat}, allowing a broader range of values for the parameters $p$ and $b$. We use a modified version of Dodson-Murphy's approach \cite{MD_New}, allowing us to deal with the inhomogeneity. The proof is also valid for the classical nonlinear Schrödinger equation ($b = 0$), extending the work in \cite{MD_New} for radial solutions in all intercritical cases.
\end{abstract}

\maketitle

\section{Introduction}

In this work, we consider the Cauchy problem for the focusing inhomogeneous nonlinear Schrö\-din\-ger equation (INLS)

\begin{equation}\label{INLS}
\begin{cases}
iu_t + \Delta u + |x|^{-b}|u|^{p-1}u=0,\\
			u(x,0) = u_0(x) \in H^1(\mathbb{R}^N), 
\end{cases}
\end{equation}

as well as its homogeneous version (NLS)

\begin{equation}\label{NLS}
\begin{cases}
iu_t + \Delta u + |u|^{p-1}u=0,\\
			u(x,0) = u_0(x) \in H^1(\mathbb{R}^N), 
\end{cases}
\end{equation}

where $u: \mathbb{R}^N\times\mathbb{R} \to \mathbb{C}$, $N > 2$, $0 \leq b < 2$, and
\begin{equation}\label{cond_p}
	 1 + \frac{4-2b}{N} < p < 1+\frac{4-2b}{N-2}.
\end{equation}

The homogeneous case $b = 0$ has been extensively studied over the past decades (for a textbook treatment, we refer the reader to Bourgain \cite{Bo99}, Cazenave \cite{cazenave}, Linares-Ponce \cite{LiPo15}, Tao \cite{TaoBook}).

The inhomogeneous version of the nonlinear Schrödinger equation arises as a model in optics, in the form
\begin{equation}
	iu_t + \Delta u + V(x)|u|^{p-1}u=0,
\end{equation}

The potential $V(x)$ accounts for the inhomogeneity of the medium. We refer to Gill \cite{Gill}, Liu and Tripathi \cite{Liu} for the physical motivation. The particular case $V(x) = |x|^{-b}$ appears naturally as a limiting case of potentials $V(x)$ that decay as $|x|^{-b}$ at infinity (Genoud and Stuart \cite{g_8}).

We briefly review the literature about \eqref{INLS} and \eqref{NLS}. It is well-known that the Cauchy problem for \eqref{NLS} is locally well-posed in $H^1(\mathbb{R}^N)$, $N \geq 1$ (Ginibre and Velo \cite{GV79}, Kato \cite{Kato87}). More precisely, given $u_0 \in H^1(\mathbb{R}^N)$, there exists $T > 0$ and a unique solution $u \in C([0,T],H^1(\mathbb{R}^N))\cap S(L^2,[0,T])$ to the NLS equation \eqref{NLS}, where $S(L^2,[0,T])$ is the intersection of all $L^2$-admissible spaces (see Definition \ref{Hs_adm} below). 

For the case $b > 0$, Genoud and Stuart \cite{g_8} proved that \eqref{INLS} is locally well-posed in $H^1(\mathbb{R}^N)$, $N \geq 1$ for $0 < b < \min\{2,N\}$. More recently, Guzm\'an \cite{Boa} established the local well-posedness of \eqref{INLS} based on Strichartz estimates. In particular, defining
\begin{equation}\label{def_b_star}
b^* = \begin{cases} \frac{N}{3}, & N \leq 3  \\
2, & N \geq 4, \\
\end{cases}    
\end{equation}

he proved that, for $N \geq 2$ and $0 < b < b^*$, the initial value problem \eqref{INLS} is locally well-posed in $H^1(\mathbb{R}^N)$. Dinh \cite{Boa_Dinh} extended Guzmán's results in dimension $N = 3$ for $0 < b < \frac{3}{2}$ and $1+\frac{4-2b}{N} < p < \frac{5-2b}{2b-1}$. Note that, in the results of Guzm\'an \cite{Boa} and Dinh \cite{Boa_Dinh}, the ranges of $b$ are more restricted than those in the results of Genoud and Stuart \cite{g_8} (mainly due to the natural restrictions on Sobolev embeddings). However, Guzmán and Dinh give more detailed information on the solutions, showing that there exists $T(\|u_0\|_{H^1})>0$ such that $u \in S(L^2,[0,T])$.

These equations are invariant under scaling. Indeed, if $u(x,t)$ is a solution to \eqref{INLS}, then
\begin{equation}
	u_\lambda(x,t) = \lambda^\frac{2-b}{p-1}u(\lambda x, \lambda^2 t), \quad \lambda > 0,
\end{equation}
is also a solution. Computing the homogeneous Sobolev norm, we obtain
\begin{equation}
	\|u_\lambda(\cdot,0)\|_{\dot{H}^s} = \lambda^{s-\frac{N}{2}-\frac{2-b}{p-1}}\|u_0\|_{\dot{H}^s}.
\end{equation}

The Sobolev index which leaves the scaling symmetry invariant is called the \textit{critical index} and is defined as
\begin{equation}
s_c = \frac{N}{2} - \frac{2-b}{p-1}.
\end{equation}

Note that the condition \eqref{cond_p} is equivalent to $0 < s_c < 1$.

Solutions to the Cauchy problem \eqref{INLS} conserve mass $M[u]$ and energy $E[u]$, defined by
\begin{equation}\label{def_mass}
M\left[u(t) \right] = \int |u(t)|^2 dx = M[u_0],
\end{equation}
\begin{equation}\label{def_energy}
E\left[u(t) \right] = \frac{1}{2}\int |\nabla u(t)|^2 dx - \frac{1}{p+1} \int |x|^{-b}|u(t)|^{p+1} dx = E[u_0].
\end{equation}

Note that mass and energy are not scale-invariant quantities when $0 < s_c < 1$. However, the interpolation quantity $M[u_0]^{1-s_c}E[u_0]^{s_c}$ defined by Holmer and Roudenko \cite{holmer2007blow} is invariant under scaling, and plays a crucial role in the description of global behavior of solutions to \eqref{INLS}.

The global behavior of $H^1(\mathbb{R}^N)$ solutions to \eqref{INLS} is related to the existence of \textit{standing waves} $u(x,t) = e^{it}\phi(x)$, where $\phi \in H^1(\mathbb{R}^N)$ satisfies the elliptic equation
\begin{equation}\label{def_Q}
	\Delta \phi - \phi + |x|^{-b}|\phi|^{p-1}\phi = 0.
\end{equation}

Standing waves of particular interest are given by solutions of \eqref{def_Q} which are positive and radial, also known as \textit{ground states}. Questions about existence and uniqueness of ground states were answered in Berestycki and Lions \cite{BL83}, Gidas et al. \cite{Gidas81}, Kwong \cite{Kwong89} for the case $b=0$. For the inhomogeneous case, existence of ground state was proved in Genoud \cites{g_5, g_6}, Genoud and Stuart \cite{g_8}, while uniqueness was handled in Yanagida \cite{g_19}, Genoud \cite{g_7}. Existence and uniqueness of $Q$, the radial, positive solution to \eqref{def_Q} hold for $N \geq 1$ and $0 \leq b < \min\{2,N\}$. 
\begin{re}\label{re_Q}
It is worth mentioning that $E(Q) > 0$ if $0 < s_c < 1$ and $Q$ decays exponentially.
\end{re}

Before stating our main result, we give the scattering criterion, which was first proved for the $3d$ cubic NLS equation by Tao \cite{Tao_Scat}.
\begin{theorem}[Scattering criterion]\label{scattering_criterion}
Let  $N >2$, $1+\frac{4-2b}{N} < p < 1+\frac{4-2b}{N-2}$ and $0 \leq b < 2$. Consider a spherically symmetric $H^1(\mathbb{R}^N)$ solution $u$ to \eqref{INLS} defined on $[0,+\infty)$ and assume the a priori bound 
\begin{equation}\label{E}
\displaystyle\sup_{t \in [0,+\infty)}\left\|u(t)\right\|_{H^1_x} = E < +\infty.
\end{equation}

There exist constants $R > 0$ and $\epsilon>0$ depending only on $E$, $N$, $p$ and $b$ (but never on $u$ or $t$) such that if
\begin{equation}\label{scacri}
\liminf_{t \rightarrow +\infty}\int_{B(0,R)}|u(x,t)|^2 \, dx \leq \epsilon^2,
\end{equation}
then there exists a function $u_+ \in H^1(\mathbb{R}^N)$ such that 
$$
\lim_{t \rightarrow +\infty}\left\|u(t)-e^{it\Delta}u_+\right\|_{H^1(\mathbb{R}^n)} = 0,
$$ 
i.e., $u$ scatters forward in time in $H^1(\mathbb{R}^N)$.
\end{theorem}

\begin{re} The notation $N > 2$ instead of $N \geq 3$ is intentional, since we allow $N$ to be arbitrarily close to $2$. At least in the radial case, it is possible to define Sobolev spaces with non-integer $N$, as in this case the dimension becomes just a parameter. It is also mathematically convenient, as this flexibility is useful in some harder proofs. We mention here the work of Kopell and Landman \cite{KL95} in which they constructed a \textit{blow-up profile} for equation \eqref{NLS} in the cubic case when the dimension $N$ is exponentially asymptotically close to 2. In \cite{MRS10},
Merle, Raphael and Szeftel constructed stable blow-up solutions in the cubic case when $d \gtrapprox 2$. Later, Rottshafer and Kaper \cite{RK02} improved the construction in \cite{KL95} to allow the dimension to be polynomially close to 2.
\end{re}

The criterion above is used to prove scattering in $H^1$ below the mass-energy threshold, as in the following theorem. We emphasize that the main aim of this paper is to show that a different approach, based on Dodson-Murphy's method, instead of the classic Kenig-Merle's concentration-compactness-rigidity technique, can be applied to the INLS equation. Moreover, our method extends the range of parameters in which scattering can be proved.
\begin{theorem}\label{teo1} Let $N > 2$,  $1+\frac{4-2b}{N} < p < 1+\frac{4-2b}{N-2}$, $0 \leq b < 2$, and $u_0 \in H^1_{rad}$ be such that 
$$M[u_0]^\frac{1-s_c}{s_c}E[u_0] < M[Q]^\frac{1-s_c}{s_c}E[Q]$$ and $$\|u_0\|_{L^2}^\frac{1-s_c}{s_c}\|\nabla u_0\|_{L^2} < \|Q\|_{L^2}^\frac{1-s_c}{s_c}\|\nabla Q\|_{L^2}.$$ 
Then the solution $u(t)$ to \eqref{INLS} is defined on $\mathbb{R}$ and scatters in $H^1$ in both time directions.
\end{theorem}
\begin{re} The above result is known for $b = 0$ and proved in Holmer and Roudenko \cite{HR_Scat}  Duyckaerts et al. \cite{DHR_Scat}, Fang et al. \cite{FXC_Scat}, Guevara \cite{Guevara}.

The case $b >0$ is considered by Farah and Guzm\'an \cite{FG_Scat} with the assumption $0 < b < \min\{N/3,1\}$, for $N \geq 2$. In the theorem above, not only we employ a new method to prove scattering, but we actually extend the range of $b$ in dimensions $N > 2$, allowing $0 < b < ]min\{N/2,2\}$ in this case. Moreover, we extend the range of $p$ in the case $N = 3$. Indeed, the result proved in Farah and Guzm\'an \cite{FG_Scat} considered $p < 4-2b$, while here we allow $p$ to be in all the intercritical range for the 3d case.
\end{re}

\begin{re}
The proofs in \cites{HR_Scat, DHR_Scat, FXC_Scat, FG_Scat, Guevara} use the so-called concentration-com\-pact\-ness-rigidity approach, pionereed by Kenig and Merle \cite{KM_Glob} in the context of the energy-critical ($s_c = 1$) NLS equation. More recently, Dodson and Murphy \cite{MD_New} developed a new approach, based on Tao's scattering criterion in \cite{Tao_Scat} and on \textit{Virial/Morawetz estimates}. We develop here a modification of Dodson-Murphy's approach, replacing $L_t^2 W^{1,2^*}_x$ estimates by local-in-time Strichartz estimates which, together with small data theory, makes it possible to handle the inhomogeneity. Since our estimates also hold in the case $b = 0$, we immediately extend the proof in \cite{MD_New}, to $0 < s_c < 1$, $N >2$ (see also Arora \cite{AndyScat}). In lower dimensions, this approach fails due to the slow decay on time of the Schr\"odinger operator $e^{it\Delta}$.
\end{re}

This paper is organized as follows: in the next section, we introduce some notation and basic estimates. In Section $3$, we prove the scattering criterion (Theorem \ref{scattering_criterion}). In Section $4$, we apply this criterion, together with Morawetz/Virial estimates to prove Theorem \ref{teo1}.

\section{Notation and basic estimates}

We denote by $p'$ the Holder's conjugate of $p \geq 1$. We use $X \lesssim Y$  to denote $X \leq C Y$, where the constant $C$ only depends on the parameters (such as $N$, $p$, $b$, as well as $E$ in \eqref{E}) and exponents, but never on $u$ or on $t$. The notations $a^+$ and $a^-$ denote, respectively, $a+\eta$ and $a-\eta$, for a fixed $0 < \eta \ll 1$. We use $p^*$ to denote the critical exponent of the Sobolev embedding $H^1 \hookrightarrow L^{p^*}$, that is, $p^* = 2N/(N-2)$ if $N > 2$, and $p^* = +\infty$ if $N \leq 2$.

\begin{de}\label{Hs_adm} If $N \geq 1$ and $s \in (-1,1)$, the pair $(q,r)$ is called $\dot{H}^s$\textit{-admissible} if it satisfies the condition
\begin{equation}\label{hs_adm_eq}
\frac{2}{q} = \frac{N}{2}-\frac{N}{r}-s,    
\end{equation}
where
$$
2 \leq q,r \leq \infty, \text{ and } (q,r,N) \neq (2,\infty,2).
$$
In particular, if $s=0$, we say that the pair is $L^2$-admissible.
\end{de}

\begin{de}\label{As}
Given $N> 2$, consider the set
\begin{equation}
    \mathcal{A}_0 = \left\{(q,r)\text{ is } L^2\text{-admissible} \left|
    \, 2 \leq r \leq \frac{2N}{N-2}
    \right.
    \right\}.
\end{equation}
For $N>2$ and $s \in (0,1)$, consider also
\begin{equation}
    \mathcal{A}_s = \left\{(q,r)\text{ is } \dot{H}^{s}\text{-admissible} \left|\,
    \left(\frac{2N}{N-2s}\right)^+ \leq r \leq \left(\frac{2N}{N-2}\right)^-
    \right.
    \right\}
\end{equation}
and
\begin{equation}
    \mathcal{A}_{-s} = \left\{(q,r) \text{ is } \dot{H}^{-s}\text{-admissible} \left|
    \left(\frac{2N}{N-2s}\right)^+ \leq r \leq \left(\frac{2N}{N-2}\right)^-
    \right.\right\}.
\end{equation}
We define the following Strichartz norm
\begin{equation}
	\|u\|_{S(\dot{H}^s,I)} = \sup_{(q,r)\in \mathcal{A}_s}\|u\|_{L_I^qL_x^r},	
\end{equation}

and the dual Strichartz norm
\begin{equation}
	\|u\|_{S'(\dot{H}^{-s},I)} = \inf_{(q,r)\in \mathcal{A}_{-s}}\|u\|_{L_I^{q'}L_x^{r'}}.	
\end{equation}
If $s=0$, we shall write $S(\dot{H}^0,I) = S(L^2,I)$ and $S'(\dot{H}^0,I) = S'(L^2,I)$. If $I=\mathbb{R}$, we will often omit $I$.
\end{de}


\subsection{Strichartz Estimates }

In this work, we use the following versions of the Strichartz estimates:

\textit{\uline{The standard Strichartz estimates} (Cazenave \cite{cazenave}, Keel and Tao \cite{KT98}, Foschi \cite{Foschi05}})

\begin{equation}\label{S1}
\|e^{it\Delta}f\|_{S(L^2)} \lesssim\|f\|_{L^2},
\end{equation}

\begin{equation}\label{S2}
\|e^{it\Delta}f\|_{S(\dot{H}^s)} \lesssim\|f\|_{\dot{H}^s},
\end{equation}
\noeqref{S2}
\begin{equation}\label{KS1}
\left\|\int_\mathbb{R}e^{i(t-\tau)\Delta}g(\cdot,\tau) \, d\tau\right\|_{S(L^2,I)} + \left\|\int_0^te^{i(t-\tau)\Delta}g(\cdot,\tau) \, d\tau\right\|_{S(L^2,I)}\lesssim\|g\|_{S'(L^2,I)}.
\end{equation}
\noeqref{KS1}
\textit{\uline{The Kato-Strichartz estimate} (Kato \cite{Kato94}, Foschi \cite{Foschi05}})
\begin{equation}\label{KS2}
\left\|\int_\mathbb{R}e^{i(t-\tau)\Delta}g(\cdot,\tau) \, d\tau\right\|_{S(\dot{H}^s,I)} + \left\|\int_0^te^{i(t-\tau)\Delta}g(\cdot,\tau) \, d\tau\right\|_{S(\dot{H}^s,I)}\lesssim\|g\|_{S'(\dot{H}^{-s},I)}.
\end{equation}
And a \textit{\uline{local-in-time estimate}}

\begin{equation}\label{lit}
\left\|\int_a^be^{i(t-\tau)\Delta}g(\cdot,\tau) \, d\tau\right\|_{S(\dot{H}^{s},\mathbb{R})}\lesssim\|g\|_{S(\dot{H}^{-s}, [a,b])}.
\end{equation}

These relations are obtained from the decay of the linear operator (see, for instance, Linares and Ponce \cite[Lemma 4.1]{LiPo15})
\begin{equation}\label{decay}
    \|e^{it\Delta}f\|_{L^p_x} \lesssim \frac{1}{|t|^{\frac{N}{2} \left(\frac{1}{p'} - \frac{1}{p}\right)}} \|f\|_{L^{p'}_x}, \quad p \geq 2, 
\end{equation}
combined with Sobolev inequalities and interpolation. The inequalities \eqref{S1}-\eqref{KS2} are standard in the theory \cite{cazenave}. To prove \eqref{lit}, we recall the following definition.
\begin{de}
If $f \in L^{p}(\mathbb{R})$, and $0 < \alpha < 1$, define the Riesz potential of order $\alpha$ as
\begin{equation}
    I_{\alpha}f(t) = \int_{-\infty}^{+\infty}\frac{1}{|t-\tau|^{1-\alpha}}f(\tau) \, d\tau.
\end{equation}
\end{de}
The next theorem is well-known, and we refer the reader to Stein \cite[Page 119, Theorem 1]{steinbook} for a complete proof.
\begin{theorem}[Hardy-Littlewood-Sobolev] If $p$, $q>1$, $0 < \alpha < 1$ and $\frac{1}{q}+ \frac{1}{p} =\alpha$, then
\begin{equation}
    \|I_{1-\alpha}f\|_{L^{q}(\mathbb{R})} \lesssim \|f\|_{L^{p'}(\mathbb{R})}.
\end{equation}
\end{theorem}

\begin{proof}[Proof of \eqref{lit}]

For $s \in [0,1)$, let $q$, $\tilde{q}$ and $r$ be such that $(q,r)$ is an $\dot{H}^s$-admissible pair, and $(\tilde{q},r)$ is an $\dot{H}^{-s}$-admissible pair. If $s = 0$, assume additionally that $2 < q < \infty$. Consider $\alpha := (N/2) (1/r' - 1/r) = 2/\tilde{q}+s = 2/q-s$ and note that $0 < \alpha < 1$ and $\frac{1}{q}+\frac{1}{\tilde{q}}=\alpha$. From Minkowski's inequality, and the decay of the linear Schrödinger operator \eqref{decay}:
\begin{align}
\left\| \int_a^be^{i(t-\tau)\Delta}g(\cdot,\tau)\, d\tau\right\|_{L^{r}_x} &\leq \int_a^b\left\|e^{i(t-\tau)\Delta}g(\cdot,\tau) \right\|_{L^{r}_x}\, d\tau \\
&\lesssim   \int_a^b\frac{1}{|t-\tau|^\alpha}\left\|g(\tau)\right\|_{L_x^{r'}}\, d\tau\\
&=  \int_{-\infty}^{+\infty}\frac{1}{|t-\tau|^\alpha}\chi_{[a,b]}(\tau)\left\|g( \tau)\right\|_{L_x^{r'}}\, d\tau\\
&=  I_{1-\alpha}\left(\chi_{[a,b]}\left\|g\right\|_{L^{r'}_x} \right)(t).
\end{align}

From the Hardy-Littlewood-Sobolev Theorem, we get
\begin{equation}\label{Interp1}
    \left\| \int_a^be^{i(t-\tau)\Delta}g(\cdot,\tau)\, d\tau\right\|_{L^q_t L^{r}_x} \lesssim \left\|\chi_{[a,b]}\left\|g\right\|_{L^{r'}_x}\right\|_{L^{\tilde{q}'}_t} =
    \left\|g\right\|_{L^{\tilde{q}'}_{[a,b]}L^{r'}_x}. 
\end{equation}
In particular, if $s=0$, then $q = \tilde{q}$ and 
 \begin{equation}\label{Interp_s_0}
 \left\| \int_a^be^{i(t-\tau)\Delta}g(\cdot,\tau)\, d\tau\right\|_{L^q_t L^{r}_x} \lesssim  \left\|g\right\|_{L^{q'}_{[a,b]}L^{r'}_x} . 
 \end{equation}
 
Note that \eqref{Interp_s_0} also immediately holds in the case $(s,q,r) = (0,\infty,2)$. Now observe that, if $s=0$ and $g \in C^\infty_0(\mathbb{R}^{N+1})$,

\begin{align}
    \left\| \int_a^be^{i(t-\tau)\Delta}g(\cdot,\tau)\, d\tau\right\|_{L^{2}_x}^2  &= \int \int_a^be^{i(t-\tau)\Delta}g(\cdot,\tau)\, d\tau  \overline{\int_a^be^{i(t-\tau')\Delta}g(\cdot,\tau')\, d\tau'}\, dx\\
    &=\int \int_a^b g(\cdot,\tau)  \overline{\int_a^be^{i(\tau-\tau')\Delta}g(\cdot,\tau')\, d\tau'}\, d\tau\,dx\\
    &\leq \int_a^b \|g(\tau)\|_{L^{r'}_x} \left\|\int_a^be^{i(\tau-\tau')\Delta}g(\cdot,\tau')\, d\tau'\right\|_{L^r_x} \, d\tau\\
    &\leq \|g\|_{L^{q'}_{[a,b]}L^{r'}_x} \left\|\int_a^be^{i(\tau-\tau')\Delta}g(\cdot,\tau')\, d\tau'\right\|_{L^q_\tau L^r_x}\\
    &\label{Interp2}\lesssim \|g\|_{L^{q'}_{[a,b]}L^{r'}_x}^2.
\end{align}
 
Therefore, as in Kato \cite[Theorem 2.1]{Kato94}, we can interpolate \eqref{Interp1} and \eqref{Interp2} and use a density argument to obtain \eqref{lit}.
\end{proof}

\subsection{Other useful estimates}

We start recalling a couple of useful estimates for radial functions. The first one is the so-called Strauss lemma. The second estimate is a Gagliardo-Nirenberg-type estimate, which is an immediate consequence of the first inequality.

\begin{lem}[Strauss \cite{Strauss77}] If $f \in H^1_{rad}(\mathbb{R}^N)$, $N\geq 2$, then, for any $R>0$,
\begin{equation}\label{Strauss}
	\|f\|_{L^\infty_{\left\{|x|\geq R\right\}}} \lesssim R^{-\frac{N-1}{2}}\|f\|_{H^1}.
\end{equation}
\end{lem}

\begin{coro}\label{Radial_GN} If $f \in H^1_{rad}(\mathbb{R}^N)$, $N\geq 2$, then, for any $R>0$, \begin{equation}
\|f\|^{p+1}_{L^{p+1}_{\left\{|x|\geq R\right\}}} \lesssim R^{-\frac{(N-1)(p-1)}{2}} \|f\|_{H^1}^{p+1}.
\end{equation}
\end{coro}

In what follows we also use the following standard estimates.

\begin{lem}[{See Guzm\'an \cite[Section 4]{Boa}}]\label{lem_guz} Let $N > 2$, $u, v \in C^\infty_0(\mathbb{R}^{N+1})$, $1+\frac{4-2b}{N} < p < 1+\frac{4-2b}{N-2}$ and $0 \leq b <\min\{N/2,2\}$. Then there exists $0 \leq \theta = \theta(N,p,b) \ll p-1$ such that the following inequalities hold

\begin{align}
\label{L1}\||x|^{-b}|u|^{p-1}u\|_{L^\infty_IL^{r}_x} &\lesssim \|u\|_{L^\infty_I H^1_x}^p, \quad  1 \leq r < \frac{2N}{N+2},\\
\label{dual_s}\left\||x|^{-b}|u|^{p-1}u\right\|_{S'(\dot{H}^{-s_c},I)}&\lesssim\left\|u\right\|^\theta_{L^\infty_tH^1_x}\left\|u\right\|^{p-\theta}_{S(\dot{H}^{s_c},I)},\\
\label{dual_l}\left\||x|^{-b}|u|^{p-1}u\right\|_{S'\left(L^2, I\right)}&\lesssim\left\|u\right\|^\theta_{L^\infty_tH^1_x}\left\|u\right\|^{p-1-\theta}_{S\left(\dot{H}^{s_c}, I\right)}\left\|u\right\|_{S\left(L^2, I\right)},\\
\label{dual_grad_l}\left\|\nabla\left(|x|^{-b}|u|^{p-1}u\right)\right\|_{S'\left(L^2, I\right)}
&\lesssim\left\|u\right\|^\theta_{L^\infty_tH^1_x}\left\|u\right\|^{p-1-\theta}_{S
\left(\dot{H}^{s_c}, I\right)}\left\|\nabla u\right\|_{S\left(L^2, I\right)}^{}.\\
\end{align}
\end{lem}

\begin{proof}
Inequality \eqref{L1} follows immediately from H\"older and Sobolev inequalities. To prove the remaining inequalities, consider the exponents
\begin{align}
    \hat{q} = \frac{4(p-1)(p+1)}{(p-1)[N(p-1)+2b]-\theta[N(p-1)-4+2b]},\,\,
    \hat{r} = \frac{N(p-1)(p+1)}{(p-1)(N-b)-\theta(2-b)},\\
    \tilde{a} = \frac{2(p-1)(p+1-\theta)}{(p-1)[N(p-\theta)-2+2b)]-(4-2b)(1-\theta)},\,\,
    \hat{a} = \frac{2(p-1)(p+1-\theta)}{4-2b-(N-2)(p-1)}.
\end{align}
Choosing $\theta = 0$ if $b = 0$, and $0<\theta \ll 1$ if $b > 0$, we have that $(\hat{q},\hat{r}) \in \mathcal{A}_0$, $(\hat{a},\hat{r}) \in \mathcal{A}_{s_c}$ and $(\tilde{a},\hat{r}) \in \mathcal{A}_{-s_c}$. By H\"older and Sobolev inequalities (see \cite[Lemmas 4.1 and 4.2]{Boa} for details), we have
\begin{equation}\label{base_l_s}
    \||x|^{-b}|u|^{p-1}v\|_{L^{\hat{r}'}} \lesssim \|u\|_{H^1_x}^\theta\|u\|^{p-1-\theta}_
    {L^{\hat{r}_x}}\|v\|_{L^{\hat{r}}_x},
\end{equation}
so that \eqref{dual_s} and \eqref{dual_l} follow. 

Consider now \eqref{dual_grad_l}. If $b = 0$, then it  follows directly from \eqref{base_l_s}. For $b > 0$, define the pairs
\begin{align}
    \bar{q} &= \frac{4(p-1)(p-\theta)}{(p-1)[N(p-1)+2b-2]-\theta[N(p-1)-4+2b]},\\
    \bar{r} &= \frac{2N(p-1)(p-\theta)}{(p-1)(N+2-2b)-\theta(4-2b)},\\
    \bar{a} &= \frac{4(p-1)(p-\theta)}{4-2b-(N-2)(p-1)}.
\end{align}

It is immediate to check that $(2,2N/(N-2)$, $(\bar{q},\bar{r}) \in \mathcal{A}_0$, and that $(\bar{a}, \bar{r})\in \mathcal{A}_{s_c}$. Let $B$ be the unit ball centered at the origin, $B^c = \mathbb{R}^N\backslash B$ and let $A$ denote $B$ or $B^c$. Since
\begin{equation}
    |\nabla(|x|^{-b}|u|^{p-1}u)| \lesssim |x|^{-b}|u|^{p-1}|\nabla u| + |x|^{-b}|x|^{-1}(|u|^{p-1}|u|),
\end{equation}
we estimate, by H\"older inequality
\begin{equation}\label{eq_scat_lem_guz}
    \|\nabla(|x|^{-b}|u|^{p-1}u)\|_{L^{\frac{2N}{N+2}}_A} \lesssim \| |x|^{-b}\|_{L^{r_1}_A}\left(\||u|^{p-1}\nabla u\|_{L^{r_2}} + \||x|^{-1}|u|^{p-1}u\|_{L^{r_2}}\right),
\end{equation}
where we choose
\begin{equation}
     \frac{1}{r_1} = \frac{b}{N}+l, \text{ with } l := \begin{cases}
        \frac{\theta(1-s_c)}{N}, &\text{ if } A = B,\\
        -\frac{\theta s_c}{N}, &\text{ if } A = B^c,
    \end{cases}
\end{equation}
and
    \begin{equation}
        \frac{1}{r_2} = \frac{N+2}{2N}-\frac{1}{r_1}.
    \end{equation}
Since $1 < \frac{2N}{N+2-2b} < N$ for $N > 2$ and $0 < b < N/2$, if we choose $\theta$ (and thus $l$) small enough, we conclude that $\||x|^{-b}\|_{L^{r_1}_A} < +\infty$, and that $1 < r_2 < N$. In view of Hardy's inequality (see \cite{kufner1990hardy}),
\begin{equation}
    \int |\nabla f|^r \geq \left(\frac{N-r}{r}\right)^r\int \frac{|f|^r}{|x|^r}, \quad f \in W^{1,r}(\mathbb{R}^N), \,1 < r < N,
\end{equation}
we have
\begin{equation}
  \||x|^{-1}|u|^{p-1}u\|_{L^{r_2}} \lesssim  \|\nabla (|u|^{p-1} u)\|_{L^{r_2}} \lesssim \||u|^{p-1}\nabla u\|_{L^{r_2}}.
\end{equation}

Therefore, \eqref{eq_scat_lem_guz} becomes
\begin{equation}
    \|\nabla(|x|^{-b}|u|^{p-1}u)\|_{L^{\frac{2N}{N+2}}} \lesssim \||u|^{p-1}\nabla u\|_{L^{r_2}}.
\end{equation}

Now, by splitting 
\begin{equation}
     \frac{1}{r_2} = \underbrace{\theta\left(\frac{1}{2}-\frac{s_c}{N} \right)-l}_{\frac{1}{r_3}}
     +\underbrace{\frac{p-1-\theta}{\bar{r}}}_{\frac{1}{r_4}} 
     + \underbrace{\frac{1}{\bar{r}}}_{\frac{1}{r_5}},
\end{equation}
it is easy to see that $2 \leq \theta r_3 \leq 2N/(N-2)$. By H\"older and Sobolev inequalities
\begin{equation}
    \||u|^{p-1}\nabla u\|_{L^{r_2}} \lesssim \|u\|_{L^{\theta r_2}}^{\theta}\|u\|_{L^{\bar{r}}}^{p-\theta}\|\nabla u\|_{L^{\bar{r}}} \lesssim \|u\|_{H^1}^{\theta}\|u\|_{L^{\bar{r}}}^{p-\theta}\|\nabla u\|_{L^{\bar{r}}}.
\end{equation}
Therefore, by H\"older inequality on the time variable:
\begin{equation}
    \|\nabla (|x|^{-b}|u|^{p-1}u)\|_{L^2_t L_x^{\frac{2N}{N+2}}} \lesssim \|u\|_{L^\infty_t H^1_x}^{\theta}\|u\|_{L^{\bar{a}}_t L^{\bar{r}}_x}^{p-1-\theta} \|\nabla u\|_{L^{\bar{q}}_t L^{\bar{r}}_x},
\end{equation}
which finishes the proof of the lemma.
\end{proof}

\begin{re}\label{re_guz}
Inequalities \eqref{dual_s}-\eqref{dual_grad_l} were proved in \cite{Boa} for $0 < b < b^*$ and with the additional restriction $p < 4-2b$ instead of $p < 5-2b$ in the 3d case. The proof we give here extends the range of $p$ and $b$ to the whole range where local well-posedness is proved. We expect that Lemma \ref{lem_guz} can be used to extend the results in \cite{Boa} using the concentration-compactness-rigidity tecnique.
\end{re}
The next lemma was proved in \cite{Boa} with the same restrictions mentioned in Remark \ref{re_guz}. In view of Lemma \ref{lem_guz}, the proof in \cite{Boa} immediately extends to the new range of $p$ and $b$.
\begin{lem}[{Small data theory, see Guzm\'an \cite[Theorem 1.8]{Boa}}]\label{small_data} Let $N \geq 1$, $1+\frac{4-2b}{N} < p < 1+\frac{4-2b}{N-2}$ and $0 \leq b <\min\{N/2,2\}$. Suppose $\|u_0\|_{H^1} \leq E$. Then there exists $\delta_{sd} = \delta_{sd}(E) > 0$ such that if 
\begin{equation}
	\|e^{it\Delta}u_0\|_{S(\dot{H}^{s_c},[0,+\infty))} \leq \delta_{sd},
\end{equation}

then the solution $u$ to \eqref{INLS} with initial condition $u_0 \in H^1(\mathbb{R}^N)$ is globally defined on $[0,+\infty)$. Moreover,
\begin{equation}
		\|u\|_{S(\dot{H}^{s_c},[0,+\infty))} \leq 2 	\|e^{it\Delta}u_0\|_{S(\dot{H}^{s_c},[0,+\infty))},
\end{equation}
and
\begin{equation}
	\|u\|_{S(L^2, [0,+\infty))}+\|\nabla u\|_{S(L^2, [0,+\infty))} \lesssim \|u_0\|_{H^1}.
\end{equation}
\end{lem}

\section{Proof of the scattering criterion}
We start this section with a remark.
\begin{re}\label{delta}
Under Definition \ref{As}, there exists a small $\delta >0$ (possibly depending on $N$, $p$, $s$ and $b$) such that, for a fixed $0 < s < 1$
$$
2 + \delta \leq r \leq p^*-\delta, \text{ and } 
$$
$$
 2+\delta \leq \frac{2}{1-s} < q \leq \frac{1}{\delta} ,
$$
for any pair $(q,r) \in \mathcal{A}_s$. 
\end{re}
 For $N > 2$, fix the parameters $$\alpha = {\frac{\delta(2+\delta)}{(p^*-\delta)(p^*-2)}} >0$$ and $$\gamma = \min\left\{{\frac{\delta(p-\theta)}{(p^*-\delta)(p^*-2)}}, \frac{\alpha(N-2) }{4}\right\}>0,$$ 
 Where $0 \leq \theta \ll p-1$ is given in Lemma \ref{lem_guz}. The following result is the key to prove Theorem \ref{scattering_criterion}.
 \begin{lem}\label{linevo}
Let  $N > 2$, $1+\frac{4-2b}{N} < p < 1+\frac{4-2b}{N-2}$, $0 \leq b < 2$ and $u$ be a radial $H^1(\mathbb{R}^N)$-solution to \eqref{INLS} satisfying \eqref{E}. If $u$ satisfies \eqref{scacri} for some $0 < \epsilon < 1$, then there exists $T > 0$ such that the following estimate is valid
\begin{equation}
\left\|e^{i(\cdot-T)\Delta}u(T)\right\|_{S\left(\dot{H}^{s_c}, [T, +\infty)\right)}  \lesssim\epsilon^\gamma.
\end{equation}
\end{lem}
 \begin{proof}
 From \eqref{S2}, there exists $T_{0} > \epsilon^{-\alpha}$ such that
\begin{equation}\label{T0}
\left\|e^{it\Delta}u_0\right\|_{S\left(\dot{H}^{s_c}, [T_0,+\infty) \right)} \leq \epsilon^{\gamma}.
\end{equation}

For $T\geq T_0$ to be chosen later, define  $I_1 :=\left[T-\epsilon^{-\alpha}, T\right]$, $I_2 := [0, T-\epsilon^{-\alpha}]$  and let $\eta$ denote a smooth, spherically symmetric function which equals $1$ on $B(0, 1/2)$ and $0$ outside $B(0,1)$. For any $R > 0$ use $\eta_R$ to denote the rescaling $\eta_R(x) := \eta(x/R)$. 

From Duhamel's formula
\begin{equation}
    u(T) = e^{iT\Delta}u_0 + \int_0^{T}e^{i(T-s)\Delta}|x|^{-b}|u|^{p-1}u(s) \, ds,
\end{equation}
we obtain
\begin{equation}
e^{i(t-T)\Delta}u(T)  = e^{it\Delta}u_0 + F_1 + F_2,
\end{equation}
where, for $i = 1,2,$
$$
F_i = \int_{I_i} e^{i(t-s)\Delta}|x|^{-b}|u|^{p-1}u(s) \,ds.
$$
We refer to $F_1$ as the ``recent past", and to $F_2$ as the ``distant past". By \eqref{T0}, it remains to estimate $F_1$ and $F_2$.

\textbf{Step 1. Estimate on recent past.}

By hypothesis \eqref{scacri}, we can fix $T\geq T_0$ such that
\begin{equation}\label{mass}
\int \eta_R(x)\left|u(T,x)\right|^2dx\lesssim \epsilon^2.
\end{equation}

Given the relation  (obtained by multiplying \eqref{INLS} by $\eta_R\bar{u}$ , taking the imaginary part and integrating by parts, see Tao \cite[Section 4]{Tao_Scat} for details)

$$
\partial_t\int \eta_R|u|^2\, dx = 2\Im\int\nabla\eta_R \cdot \nabla u \bar{u},
$$
we have, from \eqref{E}, for all times,
$$
\left| \partial_t \int \eta_R(x)|u(t,x)|^2dx\right| \lesssim \frac{1}{R},
$$

so that,  by \eqref{mass}, for $t \in I_1$,
\begin{equation}
    \int \eta_R(x)\left|u(t,x)\right|^2dx\lesssim \epsilon^2+\frac{\epsilon^{-\alpha}}{R}.
\end{equation}

If $R > \epsilon^{-(\alpha+2)}$, then we have $\left\| \eta_Ru\right\|_{L^\infty_{I_1}L^2_x} \lesssim 
\epsilon
$.

Let  $(q, r)\in \mathcal{A}_{s_c}$ . Recalling that $2 + \delta \leq r \leq p^*-\delta$ (see Remark \ref{delta}), using interpolation and Sobolev inequalities and the decay of the $L^\infty$ norm of radial functions outside the ball \eqref{Strauss}, we get
\begin{align}\label{I1}
\|u\|_{L^\infty_{I_1}L^{r}_x}
&\lesssim  \left\| \eta_Ru\right\|^{\frac{2\left(p^*-r\right)}{r\left(p^*-2\right)}}_{L^\infty_{I_1}L^2_x} \left\| \eta_R u\right\|_{L^\infty_{I_1}L^{p^*}_x}^{1-\frac{2\left(p^*-r\right)}{r\left(p^*-2\right)}} + \left\|(1-\eta_R)u\right\|^{\frac{r-2}{r}}_{L^\infty_{I_1}L_x^\infty}\left\|(1-\eta_R)u\right\|^\frac{2}{r}_{L^\infty_{I_1}L^2_x}
\nonumber\\
&\lesssim  \epsilon^{\frac{2\left(p^*-r\right)}{r\left(p^*-2\right)}}
\left\| u\right\|_{L^\infty_{I_1}L^{p^*}_x}^{1-\frac{2\left(p^*-r\right)}{r\left(p^*-2\right)}} + R^{-\frac{N-1}{2}\left( \frac{r-2}{r}\right)}\|u\|_{L_t^{\infty}H_x^1}^{\left( \frac{r-2}{r}\right)}\left\|u_{0}\right\|^\frac{2}{r}_{L^2_x}\\
&\lesssim \epsilon^{\frac{2\delta}{(p^*-\delta)(p^*-2)}}+ R^{-\frac{N-1}{2} \frac{\delta}{p^*-\delta}}
\lesssim \epsilon^{\frac{2\delta}{(p^*-\delta)(p^*-2)}},
\end{align}

if $R$ is large enough.
Note that, in the penultimate step, we used the $H^1\hookrightarrow L^{p^*}$ embedding. Using the local-in-time Strichartz estimate \eqref{lit}, together with estimates \eqref{dual_s} and \eqref{I1}, we bound

\begin{align}
\left\| \int_{I_1} e^{i(t-s)\Delta}|x|^{-b}|u|^{p-1}u(s) \,ds\right\|_{S(\dot{H}^{s_c},[T,+\infty))} &\leq ||\, |x|^{-b}|u|^{p-1}u ||_{S'(\dot{H}^{-s_c},I_1)}\\
&\hspace{-5cm}\leq  \|u\|^\theta_{L^\infty_tH_x^1} \|u\|^{p-\theta}_{S(\dot{H}^{s_c},I_1)}= \left\|u\right\|^\theta_{L^\infty_tH_x^1} \,\,  \sup_{(q,r)\in \mathcal{A}_{s_c}}\|u\|^{p-\theta}_{L^{q}_{I_1}L^{r}_x}\\
&\hspace{-5cm}\leq \left\|u\right\|^\theta_{L^\infty_tH_x^1} \,\,  \sup_{2+\delta \leq r \leq p^*-\delta}\|u\|^{p-\theta}_{L^{\infty}_{I_1}L^{r}_x}\epsilon^{-\alpha\left(\frac{p-\theta}{q}\right)}\\
&\hspace{-5cm}\leq   \left\|u\right\|^\theta_{L^\infty_tH_x^1} 
\epsilon^{\frac{2\delta}{(p^*-\delta)(p^*-2)}
({p-\theta})}\epsilon^{-\alpha\left(\frac{p-\theta}{2+\delta}\right)}\lesssim \epsilon^{\frac{\delta(p-\theta)}{(p^*-\delta)(p^*-2)}} 
,
\end{align}
where we used the definition of $\alpha >0$ and the fact that $q \geq 2+\delta$.

\textbf{Step 2. Estimate on distant past.}

Let  $(q, r) \in \mathcal{A}_{s_c}$. 
Define
\begin{equation}
	\frac{1}{c} = \left(\frac{1}{1-s_c}\right)\left[\frac{1}{q}-\delta s_c\right]
\end{equation}
and
\begin{equation}
    \frac{1}{d} = \left(\frac{1}{1-s_c}\right)\left[\frac{1}{r}-s_c\left(\frac{N-2-4\delta}{2N}\right)\right]
\end{equation}

We claim that $(c,d) \in \mathcal{A}_0$. Indeed, it is immediate to check that $(c,d)$ satisfies \eqref{hs_adm_eq} with $s=0$. Moreover, since 
\begin{equation}
    q > \frac{2}{1-s_c},
\end{equation}
we see, since $\delta > 0$ is small, that $2 < c < +\infty$, so that the pair $(c,d)$ is $L^2$-admissible. We have

$$\left\|F_2\right\|_{L_{[T,+\infty)}^{q}L_x^{r}}  
\leq 
\left\|F_2\right\|^{1-s_c}_{L_{[T,+\infty)}^{c}L_x^{d}}
\left\|F_2\right\|^{s_c}_{L_{[T,+\infty)}^{\frac{1}{\delta}}L_x^{\frac{2N}{N-2-4\delta}}}.
$$ Using Duhamel's principle, write
$$
F_2 = e^{it\Delta}\left[e^{i(-T+\epsilon^{-\alpha})\Delta}u(T-\epsilon^{-\alpha})-u(0)\right].
$$

Thus, by the Strichartz estimate \eqref{S1},
\begin{align}
\left\| F_2\right\|_{L_{[T,+\infty)}^{q}L_x^{r}}
&\leq\left\| e^{it\Delta}\left[e^{i(-T+\epsilon^{-\alpha})\Delta}u(T-\epsilon^{-\alpha})-u(0)\right]\right\|^{1-s_c}_{L_{[T,+\infty)}^{c}L_x^{d}}
\left\|F_2\right\|^{s_c}_{L_{[T,+\infty)}^{\frac{1}{\delta}}L_x^{\frac{2N}{N-2-4\delta}}}\\
&\leq\left(\left\|u\right\|_{L^\infty_t L^2_x}\right)^{1-s_c} \left\|F_2\right\|^{s_c}_{L_{[T,+\infty)}^{\frac{1}{\delta}}L_x^{\frac{2N}{N-2-4\delta}}}
\lesssim \epsilon^{\alpha \delta s_c},
\end{align}

since, by \eqref{decay} and \eqref{L1},

\begin{align}
\left\|F_2\right\|_{L_{[T,+\infty)}^{\frac{1}{\delta}}L_x^{\frac{2N}{N-2-4\delta}}}
&\lesssim \left\|\int_{I_2} |\cdot-s|^{-(1+2\delta)}\left\| |x|^{-b} |u|^{p-1}u(s) \right\|_{L^{\frac{2N}{N+2+4\delta}}_x}\, ds\right\|_{L_{[T,+\infty)}^{\frac{1}{\delta}}}\\
&\lesssim \|u\|_{L_{[T,+\infty)}^\infty H_x^1}^{p}\left\|\left(\cdot -T+\epsilon^{-\alpha}\right)^{-2\delta}\right\|_{L_{[T,+\infty)}^{\frac{1}{\delta}}}\\
&\lesssim \epsilon^{\alpha\delta}.
\end{align}

Therefore, recalling that

$$
e^{i(t-T)\Delta}u(T) = e^{it\Delta}u_0 + F_1 + F_2,
$$
we have


$$
\left\|e^{i(\cdot-T)\Delta} u(T)\right\|_{S\left(\dot{H}^{s_c}, [T,+\infty) \right)}  \lesssim\epsilon^\gamma.
$$


Hence, Lemma \ref{linevo} is proved.
\end{proof}
\begin{proof}[Proof of Theorem \ref{scattering_criterion}]

Choose $\epsilon$ is small enough so that, by Lemma \ref{linevo}, $$\left\|e^{i(\cdot)\Delta}u(T)\right\|_{S\left(\dot{H}^{s_c}, [0, +\infty)\right)}  = \left\|e^{i(\cdot-T)\Delta} u(T)\right\|_{S\left(\dot{H}^{s_c}, [T, +\infty)\right)} \leq c\epsilon^\gamma\leq \delta_{sd},
 $$
 
 where $\delta_{sd}$ is given in Lemma \ref{small_data}. Thus, by small data theory, we have 

$$\left\|u\right\|_{S\left(\dot{H}^{s_c}, [T, +\infty)\right)} \lesssim \epsilon^\gamma  \mbox{, and }\left\|(1+|\nabla|)u\right\|_{S\left(L^2,[T, +\infty)\right)} \lesssim 1.$$

Define $u_+ = e^{-iT\Delta}u(T) +i \displaystyle\int_T^{+\infty} e^{-is\Delta}|x|^{-b}|u|^{p-1}u(s) \, ds$. Using \eqref{dual_l} and \eqref{dual_grad_l}, we estimate
\begin{align}
\|u(t)-e^{it\Delta}u_+\|_{H^1_x}&=
\biggl\|\int_t^{+\infty}e^{i(t-s)\Delta}|x|^{-b} |u|^{p-1}\biggr.\biggl. u (s)\, ds\biggr\|_{H^1_x}\\
&\lesssim \biggl\|(1+|\nabla|)\int_t^{+\infty}e^{i(t-s)\Delta}|x|^{-b} |u|^{p-1}u (s)\, ds\biggr\|_{L^2_x}\\
&\lesssim \sup_{\tau \in [t,+\infty)}\biggl\|(1+|\nabla|)\int_\tau^{+\infty}e^{i(\tau-s)\Delta}|x|^{-b} |u|^{p-1}u (s)\, ds\biggr\|_{L^2_x}\\
&\lesssim 
\biggl\|\int_\tau^{+\infty}e^{i(\tau-s)\Delta}(1+|\nabla|)\left(|x|^{-b} |u|^{p-1}u (s)\right)\, ds\biggr\|_{S(L^2,[t,+\infty))}\\
&\lesssim \left\|(1+|\nabla|)\left(|x|^{-b} |u|^{p-1}u (s)\right)\right\|_{S'(L^2,[t,+\infty))}\\
&\lesssim \left\|u\right\|^{p-1-\theta}_{S\left(\dot{H}^{s_c},[t,+\infty)\right)}.
\end{align}

(Note that the same estimate ensures that $u_+\in H^1$). Hence, we conclude that 
\begin{equation}
	\lim_{t\to+\infty} \|u(t)-e^{it\Delta}u_+\|_{H^1_x} = 0
\end{equation}
as desired.
\end{proof}

\section{Proof of scattering}
We now turn to Theorem \ref{teo1}. The main idea behind the proof is to combine radial decay with a truncated Virial identity. By choosing the right weight, and using bounds given by coercivity in large balls around the origin, one can control a time-averaged $L^p$ norm on these balls. Averaging is necessary due to the lack uniform estimates in time, since we are not employing concentration-compactness as in Holmer-Roudenko \cites{HR_Scat,DHR_Scat}.

We start with the following ``trapping'' lemmas, which can be found in Farah and Guzm\'an \cite[Lemma 4.2]{FG_Scat}.
\begin{lem}[Energy trapping]\label{coercivity}
Let $N \geq 1$ and $0 < s_c < 1$. If  $$M[u_0]^\frac{1-s_c}{s_c}E[u_0] < (1-\delta)M[u_0]^\frac{1-s_c}{s_c}E[u_0]$$ for some $\delta > 0$ and $$\|u_0\|_{L^2}^\frac{1-s_c}{s_c}\|\nabla u_0\|_{L^2} \leq \|Q\|_{L^2}^\frac{1-s_c}{s_c}\|\nabla Q\|_{L^2},$$ then there exists $\delta' = \delta'(\delta) > 0$ such that
$$
\|u_0\|_{L^2}^\frac{1-s_c}{s_c}\|\nabla u_0\|_{L^2} < (1-\delta') \|Q\|_{L^2}^\frac{1-s_c}{s_c}\|\nabla Q\|_{L^2}.
$$ for all $t \in I$, where $I \subset \mathbb{R}$ is the maximal interval of existence of the solution $u(t)$ to \eqref{INLS}. Moreover, $I = \mathbb{R}$ and $u$ is uniformly bounded in $H^1$.
\end{lem}

\begin{lem}\label{coercivity2}
Suppose, for $f \in H^1(\mathbb{R}^N)$, $N \geq 1$, that $$\|f\|_{L^2}^\frac{1-s_c}{s_c}\|\nabla f\|_{L^2} < (1-\delta) \|Q\|_{L^2}^\frac{1-s_c}{s_c}\|\nabla Q\|_{L^2}.$$ Then there exists $\delta' = \delta'(\delta) > 0$  so that 
$$
 \int|\nabla f|^2 + \left(\frac{N-b}{p+1}-\frac{N}{2}\right)\int|x|^{-b}|f|^{p+1} \geq \delta'\int|x|^{-b}|f|^{p+1}.
$$
\end{lem}

From now on, we consider $u$ to be a solution to \eqref{INLS} satisfying the conditions
\begin{equation}
M[u_0]^\frac{1-s_c}{s_c}E[u_0] < M[Q]^\frac{1-s_c}{s_c}E[Q]
\end{equation} 
and
\begin{equation}
\|u_0\|_{L^2}^\frac{1-s_c}{s_c}\|\nabla u_0\|_{L^2} \le \|Q\|_{L^2}^\frac{1-s_c}{s_c}\|\nabla Q\|_{L^2}.
\end{equation}
In particular, by Lemma \ref{coercivity}, $u$ is global and uniformly bounded in $H^1$. Moreover, there exists $\delta > 0$ such that 
\begin{equation}\label{condition_u}
\sup_{t \in \mathbb{R}}\|u_0\|_{L^2}^\frac{1-s_c}{s_c}\|\nabla u(t)\|_{L^2} <(1-2\delta) \| Q\|_{L^2}^\frac{1-s_c}{s_c}\|\nabla Q\|_{L^2}
\end{equation}
In the spirit of Dodson and Murphy \cite{MD_New}, we prove a local coercivity estimate. We start with a preliminary result.
\begin{lem}\label{lem_comut} For $N \geq 1$, let $\phi$ be a smooth cutoff to the set $\{|x|\leq\frac{1}{2}\}$ and define $\phi_R(x) = \phi\left(\frac{x}{R}\right)$. If $f \in H^1(\mathbb{R}^N)$ , then
\begin{equation}\label{comut}
\int|\nabla(\phi_Rf)|^2=\int\phi_R^2|\nabla f|^2-\int \phi_R\Delta(\phi_R)|f|^2.
\end{equation}
In particular, 
\begin{equation}\label{comut2}
\left|\int|\nabla(\phi_Rf)|^2 - \int\phi_R^2|\nabla f|^2\right|\leq\frac{c}{R^2}\|f\|^2_{L^2}.
\end{equation}
\end{lem}
\begin{proof}We first calculate directly
\begin{equation}
|\nabla(\phi_R f)|^2 =\ |\nabla \phi_R f + \phi_R\nabla f|^2 = |\nabla\phi_R|^2|f|^2+2\Re(\nabla \phi_R\cdot\nabla f\,\phi_R\, \bar{f})+\phi_R^2|\nabla f|^2.
\end{equation}
Now, integrating by parts, we have
\begin{equation}
2\Re \int(\nabla \phi_R\cdot\nabla f\,\phi_R\, \bar{f}) = -\int\phi_R \Delta(\phi_R)|f|^2-\int|\nabla \phi_R|^2|f|^2.
\end{equation}
Using the last two identities, we conclude \eqref{comut}. To obtain \eqref{comut2}, we note that 
\begin{equation}
\|\phi_R\Delta(\phi_R)\|_{L^\infty} \leq \frac{c}{R^2}.
\end{equation}

\end{proof} 
\begin{lem}[Local coercivity]\label{coercivity_balls} For $N \geq 1$, let $u$ be a globally defined $H^1(\mathbb{R}^N)$-solution to \eqref{INLS} satisfying \eqref{condition_u}. There exists $ \bar{R} =\  \bar{R}(\delta, M[u_0],Q, s_c)>0$ such that, for any $R \geq \bar{R}$,
\begin{equation}
\sup_{t \in \mathbb{R}}\|\phi_Ru(t)\|^\frac{1-s_c}{s_c}_{L^2}\|\nabla(\phi_R u(t))\|_{L^2} \leq (1-\delta)\|Q\|^\frac{1-s_c}{s_c}_{L^2}\|\nabla Q\|_{L^2}.
\end{equation}
In particular, by Lemma \ref{coercivity2}, there exists $\delta' = \delta'(\delta) >0$ such that 
\begin{equation}\label{local_coercivity}
\int|\nabla ( \phi_R u(t))|^2 + \left(\frac{N-b}{p+1}-\frac{N}{2}\right)\int|x|^{-b}|\phi_R u(t)|^{p+1} \geq \delta'\int|x|^{-b}|\phi_Ru(t)|^{p+1}.
\end{equation}
\end{lem}
\begin{proof}First note that
\begin{equation}
 \|\phi_Ru(t)\|^2_{L^2} \leq \|u(t)\|^2_{L^2} = M[u_0],
\end{equation}
for all $t \in \mathbb{R}$. Thus, we only need to control the $\dot{H}^1$ term. Using Lemma \ref{lem_comut} and \eqref{condition_u}, we conclude
\begin{align}
\|\phi_Ru(t)\|^\frac{2(1-s_c)}{s_c}_{L^2}\|\phi_R u(t)\|^2_{\dot{H}^1}&\leq
M[u_0]^\frac{1-s_c}{s_c}_{L^2}\left(\|\nabla u(t)\|^2_{L^2}+\frac{c}{R^2}M[u_0]\right)\\
&<(1-2\delta)^2\|Q\|^\frac{2(1-s_c)}{s_c}_{L^2}\|\nabla Q\|^2_{L^2}+\frac{c}{R^2}M[u_0]^\frac{1}{s_c}_{L^2}.
\end{align}
Thus, by choosing $R$ large enough, depending on $\delta$, $M[u_0]$, $Q$ and $s_c$, we bound the last expression by $\left[(1-\delta)\|Q\|^\frac{1-s_c}{s_c}_{L^2}\|\nabla Q\|_{L^2}\right]^2$, which finishes the proof.

\end{proof}
We exploit the coercivity given by the previous lemma by making use of the Virial identity (see Dodson and Murphy \cite[Lemma 3.3]{MD_New}, Farah and Guzm\'an \cite[Proposition 7.2]{FG_Scat})

\begin{lem}[Virial identity]\label{Morawetz}
Let $a: \mathbb{R}^N \rightarrow \mathbb{R}$ be a smooth weight. If $|\nabla a| \in L^\infty$, define
$$
Z(t) = 2 \Im\int \bar{u} \nabla u \cdot \nabla a \, dx.
$$
Then, if $u$ is a solution to \eqref{INLS}, we have the following identity 
$$
\frac{d}{dt}Z(t) = \left(\frac{4}{p+1}-2\right)\int|x|^{-b}|u|^{p+1} \Delta a  -\frac{4b}{p+1}\int|x|^{-b-2}|u|^{p+1}x\cdot\nabla a 
$$
$$
-\int|u|^2\Delta\Delta a + 4\Re\sum_{i,j}\int a_{ij}\bar{u}_i u_j.
$$
\end{lem}
We now have all the basic tools needed to prove scattering. Let $R \gg 1$   to be determined below. We take $a$ to be a radial function satisfying
$$
a(x) = \begin{cases}|x|^2 & |x|\leq \frac{R}{2}, \\
2R|x|-R^2 & |x| > R. \\
\end{cases}
$$
In the intermediate region $\frac{R}{2} < |x| \leq R$, we impose that
$$
\partial_ra \geq 0, \,\,\, \partial_r^2a \geq 0, \,\,\,|\partial^\alpha a(x)| \lesssim_{\alpha}R |x|^{-|\alpha|+1} \,\,\, \text{for}\,\,\,|\alpha| \geq 1.
$$
Here, $\partial_r$ denotes the radial derivative, i.e., $\partial_r a = \nabla a \cdot \frac{x}{|x|}$. Note that for $|x| \leq \frac{R}{2}$, we have
$$
a_{ij} = 2 \delta_{ij}, \,\,\, \Delta a = 2N, \,\,\, \Delta \Delta a = 0,
$$
while, for $|x| > R$, we have
$$
a_{ij} = \frac{2R}{|x|}\left[\delta_{ij} - \frac{x_i}{|x|}\frac{x_j}{|x|}\right], \,\,\, \Delta a = \frac{2(N-1)R}{|x|}, \,\,\, |\Delta \Delta a(x)| \lesssim \frac{R}{|x|^3}.
$$
\begin{pro}[Virial/Morawetz estimate]\label{virial}
For $N > 2$, let $u$ be a radial $H^1$-solution to \eqref{INLS} satisfying \eqref{condition_u}. Then, for $R = R(\delta, M[u_0], Q)$ sufficiently large, and $T>0$,
$$
\frac{1}{T}\int_0^T\int_{|x|\leq R}|u(x,t)|^{p+1}\,dx\, dt \lesssim_{u,\delta} \frac{R^{b+1}}{T}+\frac{1}{R^{(2-b)\left(\frac{N-1}{N}\right)}}.
$$
\end{pro}
\begin{proof}
Choose $R\geq \bar{R}(\delta, M[u_0], Q, s_c)$ as in Lemma \ref{coercivity_balls}. We define the weight $a$ as above and define $Z(t)$ as in Lemma \ref{Morawetz}. Using Cauchy-Schwarz inequality, and the definition of $Z(t)$, we have
\begin{equation}\label{Mbound}
\sup_{t \in \mathbb{R}} |Z(t)| \lesssim R.
\end{equation}
As in Dodson and Murphy \cite[Proposition 3.4]{MD_New}, we compute
\begin{align}
\frac{d}{dt}Z(t) &=8\left[ \int_{|x|\leq\frac{R}{2}}|\nabla u|^2 + \left(\frac{N-b}{p+1}-\frac{N}{2}\right)\int_{|x|\leq \frac{R}{2}}|x|^{-b}|u|^{p+1}\right]\\
&\quad+\int_{|x| > \frac{R}{2}}\left[\left(\frac{4}{p+1}-2\right)(N-1)\Delta a  -\frac{4b}{p+1}\frac{x\cdot\nabla a }{|x|^2} \right]|x|^{-b}|u|^{p+1}\\
&\quad+\int_{ |x| > \frac{R}{2}}4\partial_r^2a|\partial_r u|^2 -\int_{ |x| > \frac{R}{2}}|u|^2 \Delta\Delta a,\\
\end{align}
where we used the radiality of $u$ and $a$. By the definition of $a$, and the fact that $\partial_r^2a \geq 0$,
\begin{align}\label{MP}\frac{d}{dt}Z(t)&\geq  8\left[ \int_{|x|\leq\frac{R}{2}}|\nabla u|^2 + \left(\frac{N-b}{p+1}-\frac{N}{2}\right)\int_{|x|\leq \frac{R}{2}}|x|^{-b}|u|^{p+1}\right]\\
&\quad-\frac{c}{R^b}\int_{|x|>\frac{R}{2}}|u|^{p+1}-\frac{c}{R^2}M[u_0].
\end{align}
Define $\phi^A$ ,  $A>0$, as a smooth cutoff to the set $\{|x| \leq \frac{1}{2}\}$ that vanishes outside the set $\{|x| \leq \frac{1}{2}+\frac{1}{A}\}$, and define $\phi_R^A(x) =\ \phi^A\left(\frac{x}{R}\right)$. We will now estimate the first term in the last inequality.

\begin{align}\label{MP2}
 \int_{|x|\leq\frac{R}{2}}&|\nabla u|^2 + \left(\frac{N-b}{p+1}-\frac{N}{2}\right)\int_{|x|\leq\frac{R}{2}}|x|^{-b}|u|^{p+1}=\nonumber\\
= &\left[ \int(\phi^A_R)^2|\nabla u|^2 + \left(\frac{N-b}{p+1}-\frac{N}{2}\right)\int(\phi^A_R)^2|x|^{-b}|u|^{p+1}\right]\nonumber\\
-&\underbrace{\left[ \int_{\frac{R}{2}<|x|\leq\frac{R}{2}+\frac{R}{A}}(\phi^A_R)^2|\nabla u|^2 + \left(\frac{N-b}{p+1}-\frac{N}{2}\right)\int_{\frac{R}{2}<|x|\leq\frac{R}{2}+\frac{R}{A}}(\phi^A_R)^2|x|^{-b}|u|^{p+1}\right]}_{I_A}\nonumber\\
=&\left[ \int|\phi^A_R\nabla u|^2 + \left(\frac{N-b}{p+1}-\frac{N}{2}\right)\int|x|^{-b}|\phi^A_Ru|^{p+1}\right]\nonumber\\
-&I_A-\underbrace{\left(\frac{N}{2}-\frac{N-b}{p+1}\right)\int\left((\phi^A_R)^{p+1}-(\phi^A_R)^{2}\right)|x|^{-b}|u|^{p+1}}_{II_A}.
\end{align}
Using Lemma \ref{lem_comut}, we can write  

\begin{align}\label{MP3}
\int&|\phi^A_R\nabla u|^2 + \left(\frac{N-b}{p+1}-\frac{N}{2}\right)\int|x|^{-b}|\phi^A_Ru|^{p+1}\geq\nonumber\\
&\int|\nabla (\phi^A_Ru)|^2 + \left(\frac{N-b}{p+1}-\frac{N}{2}\right)\int|x|^{-b}|\phi^A_Ru|^{p+1}-\frac{c}{R^2}M[u_0].
\end{align}

The inequalities \eqref{MP}, \eqref{MP2} and \eqref{MP3} can be rewritten as 
\begin{align}
\frac{d}{dt}Z(t) &\geq  8 \left[ \int|\nabla (\phi^A_Ru)|^2 + \left(\frac{N-b}{p+1}-\frac{N}{2}\right)\int|x|^{-b}|\phi^A_Ru|^{p+1}\right]\\
\label{MP4}&\quad-\frac{c}{R^b}\int_{|x|>\frac{R}{2}}|u|^{p+1}-\frac{c}{R^2}M[u_0]-8I_A - 8II_A.
\end{align}

By Corollary \ref{Radial_GN}
and by Lemma \ref{coercivity_balls},
we can write \eqref{MP4} as 
\begin{align}
\int|x|^{-b}|\phi^A_Ru(t)|^{p+1} \lesssim \frac{d}{dt}Z(t)+&\frac{1}{R^{\frac{(N-1)(p-1)}{2}+b}}+\frac{1}{R^2}+8I_A +8II_A.
\end{align}

We can now make $A \to +\infty$ to obtain $I_A + II_A \to 0$ by dominated convergence. Hence,

\begin{equation}\label{MP6}
R^{-b}\int_{|x|\leq\frac{R}{2}}|u(t)|^{p+1} \lesssim\int_{|x|\leq\frac{R}{2}}|x|^{-b}|u(t)|^{p+1} \lesssim \frac{d}{dt}Z(t)+\frac{1}{R^{\frac{(N-1)(p-1)}{2}+b}}+\frac{1}{R^2}.
\end{equation}

We finish the proof integrating over time, and using \eqref{Mbound}. We have
\begin{align}
\frac{1}{T}\int_0^T \int_{|x|\leq\frac{R}{2}}|u(t)|^{p+1} &\lesssim\frac{R^b}{T}\sup_{t\in[0,T]} |Z(t)| + \frac{1}{R^{\frac{(N-1)(p-1)}{2}}}+\frac{1}{R^{2-b}}\nonumber\\
&\lesssim \frac{R^{b+1}}{T} + \frac{1}{R^{\left(2-b\right)\frac{(N-1)}{N}}},\end{align}
since $p > 1+\frac{4-2b	}{N}$.
\end{proof}
We are now able to prove the \textit{energy evacuation}.
\begin{pro}[Energy evacuation]\label{energy_evacuation} Under the hypotheses of Proposition \ref{virial}, there exist a sequence of times $t_n \to +\infty$ and a sequence of radii $R_n \to +\infty$ such that
\begin{equation}\label{eq_energy_evac}
\lim_{n \to +\infty}\int_{|x|\leq R_n}|u(t_n)|^{p+1} = 0
\end{equation}
\end{pro}
\begin{proof}
Using Proposition \ref{virial}, choose $T_n \to +\infty$ and $R_n = T_n^{\frac{N}{3N-2+b}}$, so that 
\begin{equation}
\frac{1}{T_n}\int_0^{T_n} \int_{|x|\leq R_n}|u(t)|^{p+1}\lesssim \frac{1}{T_n^{\frac{(2-b)(N-1)}{3N-2+b}}}\to 0 \text{ as }n \to +
\infty.
\end{equation}
Therefore, by the Mean Value Theorem, there is a sequence $t_n \to +\infty$ such that \eqref{eq_energy_evac} holds. The proof is complete.
\end{proof}
Using Proposition \ref{energy_evacuation}, we can prove
Theorem \ref{teo1}. We will prove only the case $t \to +\infty$, as the case $t \to -\infty$ is entirely analogous.
\begin{proof}[Proof of Theorem \ref{teo1}]Take $t_n \to +\infty$ and $R_n \to +\infty $ as in Proposition \ref{energy_evacuation}. Fix $\epsilon > 0$ and $R>0$ as in Theorem \ref{scattering_criterion}. Choosing $n$ large enough, such that $R_n \geq R$, Hölder's inequality yields
\begin{equation}
\int_{|x|\leq R} |u(x,t_n)|^2 \lesssim R^\frac{N(p-1)}{p+1}\left(\int_{|x|\leq R_n}|u(x,t_n)|^{p+1}\right)^\frac{2}{p+1} \to 0 \text{ as } n \to +\infty.
\end{equation}
Therefore, by Theorem \ref{scattering_criterion}, $u$ scatters forward in time.
\end{proof}




   \bibliography{biblio.bbl}

\begin{bibdiv}
\begin{biblist}

\bib{AndyScat}{article}{
      author={Arora, Anudeep~Kumar},
       title={Scattering of radial data in the focusing {NLS} and generalized
  {H}artree equations},
        date={2019},
     journal={arXiv preprint arXiv:1904.05800},
}

\bib{BL83}{article}{
      author={Berestycki, H.},
      author={Lions, P.-L.},
       title={Nonlinear scalar field equations. {I}. {E}xistence of a ground
  state},
        date={1983},
        ISSN={0003-9527},
     journal={Arch. Rational Mech. Anal.},
      volume={82},
      number={4},
       pages={313\ndash 345},
         url={https://doi.org/10.1007/BF00250555},
      review={\MR{695535}},
}

\bib{Bo99}{book}{
      author={Bourgain, J.},
       title={Global solutions of nonlinear {S}chr\"{o}dinger equations},
      series={American Mathematical Society Colloquium Publications},
   publisher={American Mathematical Society, Providence, RI},
        date={1999},
      volume={46},
        ISBN={0-8218-1919-4},
         url={https://doi.org/10.1090/coll/046},
      review={\MR{1691575}},
}

\bib{cazenave}{book}{
      author={Cazenave, Thierry},
       title={Semilinear {S}chr\"{o}dinger equations},
      series={Courant Lecture Notes in Mathematics},
   publisher={New York University, Courant Institute of Mathematical Sciences,
  New York; American Mathematical Society, Providence, RI},
        date={2003},
      volume={10},
        ISBN={0-8218-3399-5},
         url={https://doi.org/10.1090/cln/010},
      review={\MR{2002047}},
}

\bib{Boa_Dinh}{article}{
      author={Dinh, V.},
       title={{Scattering theory in a weighted {$L^{2}$} space for a class of
  the defocusing inhomogeneous nonlinear Schr{\"o}dinger equation}},
        date={2017},
     journal={arXiv preprint arXiv:1710.01392},
}

\bib{MD_New}{article}{
      author={Dodson, Benjamin},
      author={Murphy, Jason},
       title={A new proof of scattering below the ground state for the 3{D}
  radial focusing cubic {NLS}},
        date={2017},
        ISSN={0002-9939},
     journal={Proc. Amer. Math. Soc.},
      volume={145},
      number={11},
       pages={4859\ndash 4867},
         url={https://doi.org/10.1090/proc/13678},
      review={\MR{3692001}},
}

\bib{DHR_Scat}{article}{
      author={Duyckaerts, Thomas},
      author={Holmer, Justin},
      author={Roudenko, Svetlana},
       title={Scattering for the non-radial 3{D} cubic nonlinear
  {S}chr\"{o}dinger equation},
        date={2008},
        ISSN={1073-2780},
     journal={Math. Res. Lett.},
      volume={15},
      number={6},
       pages={1233\ndash 1250},
         url={https://doi.org/10.4310/MRL.2008.v15.n6.a13},
      review={\MR{2470397}},
}

\bib{FXC_Scat}{article}{
      author={Fang, DaoYuan},
      author={Xie, Jian},
      author={Cazenave, Thierry},
       title={Scattering for the focusing energy-subcritical nonlinear
  {S}chr\"{o}dinger equation},
        date={2011},
        ISSN={1674-7283},
     journal={Sci. China Math.},
      volume={54},
      number={10},
       pages={2037\ndash 2062},
         url={https://doi.org/10.1007/s11425-011-4283-9},
      review={\MR{2838120}},
}

\bib{FG_Scat}{article}{
      author={Farah, L.~G.},
      author={Guzm{\'a}n, C.},
       title={Scattering for the radial focusing {INLS} equation in higher
  dimensions},
        date={2017},
     journal={arXiv preprint arXiv:1703.10988},
}

\bib{Foschi05}{article}{
      author={Foschi, Damiano},
       title={Inhomogeneous {S}trichartz estimates},
        date={2005},
        ISSN={0219-8916},
     journal={J. Hyperbolic Differ. Equ.},
      volume={2},
      number={1},
       pages={1\ndash 24},
         url={https://doi.org/10.1142/S0219891605000361},
      review={\MR{2134950}},
}

\bib{g_5}{thesis}{
      author={Genoud, F.},
       title={Th{\'e}orie de bifurcation et de stabilit{\'e} pour une
  {\'e}quation de schr{\"o}dinger avec une non-lin{\'e}arit{\'e} compacte},
        type={Ph.D. Thesis},
        date={2008},
}

\bib{g_6}{article}{
      author={Genoud, Fran\c{c}ois},
       title={Bifurcation and stability of travelling waves in self-focusing
  planar waveguides},
        date={2010},
        ISSN={1536-1365},
     journal={Adv. Nonlinear Stud.},
      volume={10},
      number={2},
       pages={357\ndash 400},
         url={https://doi.org/10.1515/ans-2010-0207},
      review={\MR{2656687}},
}

\bib{g_7}{article}{
      author={Genoud, Fran\c{c}ois},
       title={A uniqueness result for {$\Delta u-\lambda u+V(|x|)u^p=0$} on
  {$\Bbb R^2$}},
        date={2011},
        ISSN={1536-1365},
     journal={Adv. Nonlinear Stud.},
      volume={11},
      number={3},
       pages={483\ndash 491},
         url={https://doi.org/10.1515/ans-2011-0301},
      review={\MR{2839991}},
}

\bib{g_8}{article}{
      author={Genoud, Fran\c{c}ois},
      author={Stuart, Charles~A.},
       title={Schr\"{o}dinger equations with a spatially decaying nonlinearity:
  existence and stability of standing waves},
        date={2008},
        ISSN={1078-0947},
     journal={Discrete Contin. Dyn. Syst.},
      volume={21},
      number={1},
       pages={137\ndash 186},
         url={https://doi.org/10.3934/dcds.2008.21.137},
      review={\MR{2379460}},
}

\bib{Gidas81}{incollection}{
      author={Gidas, B.},
      author={Ni, Wei~Ming},
      author={Nirenberg, L.},
       title={Symmetry of positive solutions of nonlinear elliptic equations in
  {${\bf R}^{n}$}},
        date={1981},
   booktitle={Mathematical analysis and applications, {P}art {A}},
      series={Adv. in Math. Suppl. Stud.},
      volume={7},
   publisher={Academic Press, New York-London},
       pages={369\ndash 402},
      review={\MR{634248}},
}

\bib{Gill}{article}{
      author={Gill, Tarsem~Singh},
       title={Optical guiding of laser beam in nonuniform plasma},
        date={2000},
     journal={Pramana},
      volume={55},
      number={5-6},
       pages={835\ndash 842},
}

\bib{GV79}{article}{
      author={Ginibre, J.},
      author={Velo, G.},
       title={On a class of nonlinear {S}chr\"{o}dinger equations. {I}. {T}he
  {C}auchy problem, general case},
        date={1979},
        ISSN={0022-1236},
     journal={J. Funct. Anal.},
      volume={32},
      number={1},
       pages={1\ndash 32},
         url={https://doi.org/10.1016/0022-1236(79)90076-4},
      review={\MR{533218}},
}

\bib{Guevara}{article}{
      author={Guevara, Cristi~Darley},
       title={Global behavior of finite energy solutions to the
  {$d$}-dimensional focusing nonlinear {S}chr\"{o}dinger equation},
        date={2014},
        ISSN={1687-1200},
     journal={Appl. Math. Res. Express. AMRX},
      number={2},
       pages={177\ndash 243},
      review={\MR{3266698}},
}

\bib{Boa}{article}{
      author={Guzm\'{a}n, Carlos~M.},
       title={On well posedness for the inhomogeneous nonlinear
  {S}chr\"{o}dinger equation},
        date={2017},
        ISSN={1468-1218},
     journal={Nonlinear Anal. Real World Appl.},
      volume={37},
       pages={249\ndash 286},
         url={https://doi.org/10.1016/j.nonrwa.2017.02.018},
      review={\MR{3648381}},
}

\bib{HR_Scat}{article}{
      author={Holmer, Justin},
      author={Roudenko, Svetlana},
       title={A sharp condition for scattering of the radial 3{D} cubic
  nonlinear {S}chr\"{o}dinger equation},
        date={2008},
        ISSN={0010-3616},
     journal={Comm. Math. Phys.},
      volume={282},
      number={2},
       pages={435\ndash 467},
         url={https://doi.org/10.1007/s00220-008-0529-y},
      review={\MR{2421484}},
}

\bib{holmer2007blow}{article}{
      author={Holmer, Justin},
      author={Roudenko, Svetlana},
       title={A class of solutions to the 3{D} cubic nonlinear
  {S}chr\"{o}dinger equation that blows up on a circle},
        date={2011},
        ISSN={1687-1200},
     journal={Appl. Math. Res. Express. AMRX},
      number={1},
       pages={23\ndash 94},
      review={\MR{2782553}},
}

\bib{Kato87}{article}{
      author={Kato, Tosio},
       title={On nonlinear {S}chr\"{o}dinger equations},
        date={1987},
        ISSN={0246-0211},
     journal={Ann. Inst. H. Poincar\'{e} Phys. Th\'{e}or.},
      volume={46},
      number={1},
       pages={113\ndash 129},
         url={http://www.numdam.org/item?id=AIHPB_1987__46_1_113_0},
      review={\MR{877998}},
}

\bib{Kato94}{incollection}{
      author={Kato, Tosio},
       title={An {$L^{q,r}$}-theory for nonlinear {S}chr\"{o}dinger equations},
        date={1994},
   booktitle={Spectral and scattering theory and applications},
      series={Adv. Stud. Pure Math.},
      volume={23},
   publisher={Math. Soc. Japan, Tokyo},
       pages={223\ndash 238},
      review={\MR{1275405}},
}

\bib{KT98}{article}{
      author={Keel, Markus},
      author={Tao, Terence},
       title={Endpoint {S}trichartz estimates},
        date={1998},
        ISSN={0002-9327},
     journal={Amer. J. Math.},
      volume={120},
      number={5},
       pages={955\ndash 980},
  url={http://muse.jhu.edu/journals/american_journal_of_mathematics/v120/120.5keel.pdf},
      review={\MR{1646048}},
}

\bib{KM_Glob}{article}{
      author={Kenig, Carlos~E.},
      author={Merle, Frank},
       title={Global well-posedness, scattering and blow-up for the
  energy-critical, focusing, non-linear {S}chr\"{o}dinger equation in the
  radial case},
        date={2006},
        ISSN={0020-9910},
     journal={Invent. Math.},
      volume={166},
      number={3},
       pages={645\ndash 675},
         url={https://doi.org/10.1007/s00222-006-0011-4},
      review={\MR{2257393}},
}

\bib{KL95}{article}{
      author={Kopell, Nancy},
      author={Landman, Michael},
       title={Spatial structure of the focusing singularity of the nonlinear
  {S}chr\"{o}dinger equation: a geometrical analysis},
        date={1995},
        ISSN={0036-1399},
     journal={SIAM J. Appl. Math.},
      volume={55},
      number={5},
       pages={1297\ndash 1323},
         url={https://doi.org/10.1137/S0036139994262386},
      review={\MR{1349311}},
}

\bib{kufner1990hardy}{book}{
      author={Kufner, Alois},
      author={Opic, Bohum{\'\i}r},
       title={Hardy-type inequalities},
   publisher={Longman Scientific \& Technical},
        date={1990},
}

\bib{Kwong89}{article}{
      author={Kwong, Man~Kam},
       title={Uniqueness of positive solutions of {$\Delta u-u+u^p=0$} in
  {${\bf R}^n$}},
        date={1989},
        ISSN={0003-9527},
     journal={Arch. Rational Mech. Anal.},
      volume={105},
      number={3},
       pages={243\ndash 266},
         url={https://doi.org/10.1007/BF00251502},
      review={\MR{969899}},
}

\bib{LiPo15}{book}{
      author={Linares, Felipe},
      author={Ponce, Gustavo},
       title={Introduction to nonlinear dispersive equations},
     edition={Second},
      series={Universitext},
   publisher={Springer, New York},
        date={2015},
        ISBN={978-1-4939-2180-5; 978-1-4939-2181-2},
         url={https://doi.org/10.1007/978-1-4939-2181-2},
      review={\MR{3308874}},
}

\bib{Liu}{article}{
      author={Liu, CS},
      author={Tripathi, VK},
       title={Laser guiding in an axially nonuniform plasma channel},
        date={1994},
     journal={Physics of plasmas},
      volume={1},
      number={9},
       pages={3100\ndash 3103},
}

\bib{MRS10}{article}{
      author={Merle, Frank},
      author={Rapha\"{e}l, Pierre},
      author={Szeftel, Jeremie},
       title={Stable self-similar blow-up dynamics for slightly {$L^2$}
  super-critical {NLS} equations},
        date={2010},
        ISSN={1016-443X},
     journal={Geom. Funct. Anal.},
      volume={20},
      number={4},
       pages={1028\ndash 1071},
         url={https://doi.org/10.1007/s00039-010-0081-8},
      review={\MR{2729284}},
}

\bib{RK02}{article}{
      author={Rottsch\"{a}fer, Vivi},
      author={Kaper, Tasso~J.},
       title={Blowup in the nonlinear {S}chr\"{o}dinger equation near critical
  dimension},
        date={2002},
        ISSN={0022-247X},
     journal={J. Math. Anal. Appl.},
      volume={268},
      number={2},
       pages={517\ndash 549},
         url={https://doi.org/10.1006/jmaa.2001.7814},
      review={\MR{1896213}},
}

\bib{steinbook}{book}{
      author={Stein, Elias~M.},
       title={Singular integrals and differentiability properties of
  functions},
      series={Princeton Mathematical Series, No. 30},
   publisher={Princeton University Press, Princeton, N.J.},
        date={1970},
      review={\MR{0290095}},
}

\bib{Strauss77}{article}{
      author={Strauss, Walter~A.},
       title={Existence of solitary waves in higher dimensions},
        date={1977},
        ISSN={0010-3616},
     journal={Comm. Math. Phys.},
      volume={55},
      number={2},
       pages={149\ndash 162},
         url={http://projecteuclid.org/euclid.cmp/1103900983},
      review={\MR{0454365}},
}

\bib{Tao_Scat}{article}{
      author={Tao, Terence},
       title={On the asymptotic behavior of large radial data for a focusing
  non-linear {S}chr\"{o}dinger equation},
        date={2004},
        ISSN={1548-159X},
     journal={Dyn. Partial Differ. Equ.},
      volume={1},
      number={1},
       pages={1\ndash 48},
         url={https://doi.org/10.4310/DPDE.2004.v1.n1.a1},
      review={\MR{2091393}},
}

\bib{TaoBook}{book}{
      author={Tao, Terence},
       title={Nonlinear dispersive equations: local and global analysis},
   publisher={American Mathematical Soc.},
        date={2006},
      number={106},
}

\bib{g_19}{article}{
      author={Yanagida, Eiji},
       title={Uniqueness of positive radial solutions of {$\Delta
  u+g(r)u+h(r)u^p=0$} in {${\bf R}^n$}},
        date={1991},
        ISSN={0003-9527},
     journal={Arch. Rational Mech. Anal.},
      volume={115},
      number={3},
       pages={257\ndash 274},
         url={https://doi.org/10.1007/BF00380770},
      review={\MR{1106294}},
}

\end{biblist}
\end{bibdiv}

\end{document}